\documentclass[a4paper,11pt]{amsart}
\usepackage{latexsym, amssymb,amsmath,amsthm}

\def \Spec{\operatorname{Spec}}
\def \gr{\operatorname{gr}}

\numberwithin{equation}{section}

\newtheorem{theorem}{Theorem}[section]
\newtheorem{lemma}[theorem]{Lemma}
\newtheorem{proposition}[theorem]{Proposition}
\newtheorem{corollary}[theorem]{Corollary}
\newtheorem{definition}[theorem]{Definition}
\newtheorem{example}[theorem]{Example}
\newtheorem{remark}[theorem]{Remark}

\begin{document}
\title{Commutative Hopf structures over a loop}
\author{Hua-Lin Huang}
\address{School of Mathematics, Shandong University, Jinan
250100, China} \email{hualin@sdu.edu.cn}
\author{Gongxiang Liu}
\address{Department of Mathematics, Nanjing University, Nanjing 210093, China}
\email{gxliu@nju.edu.cn}
\author{Yu Ye}
\address{Department of Mathematics, University of Science and Technology of China,
 Hefei 230026, China}
\email{yeyu@ustc.edu.cn}
\date{}
\maketitle
\begin{abstract}
Let $k$ be an algebraically closed field of characteristic $p>0$.
For a loop $\circlearrowleft$, denote its path coalgebra by
$k\circlearrowleft$. In this paper, all the finite-dimensional
commutative Hopf algebras over the sub coalgebras of
$k\circlearrowleft$ are given. As a direct consequence, all the
commutative infinitesimal groups $\mathcal{G}$ with
dim$_{k}$Lie$(\mathcal{G})=1$ are classified. \vskip 5pt

\noindent{\bf Keywords} \ \ Path coalgebra, Unipotent group, Frobenius map \\
\noindent{\bf 2000 MR Subject Classification} \ \ 16W30, 14L15
\end{abstract}

\section{introduction}

This paper is concerned with the quiver realization of
finite-dimensional cocommutative Hopf algebras. As is well-known,
any such algebra can be viewed as the group algebra of a finite
algebraic $k$-group $\mathcal{G}$. Considerable attention has been
received by these algebraic groups.

Quivers are oriented diagrams consisting of vertices and arrows
\cite{ARS}. Due to the well-known theorem of Gabriel given in the
early 1970s, these combinatorial stuffs make the abstract elementary
algebras and their representations visible. This point of view has
since then played a central role and is generally accepted as the
starting point in the modern representation theory of associative
algebras. Naturally there is a strong desire to apply this handy
quiver tool to other algebraic structures. Such idea for Hopf
algebras appeared explicitly in \cite{Cil,Cil1,CR,GS} and was showed
to be very effective in dealing with the structures of
finite-dimensional pointed (or dually, basic) Hopf algebras when the
characteristic of the base field is 0 \cite{CR,CHYZ,GL,L,HL,FP}.

Comparing to the characteristic 0 case, there is hardly any work
dealing with the positive characteristic case by using quiver
methods, see however a recent work of Cibils, Lauve and Witherspoon
\cite{clw}. One main difficulty in the positive characteristic case
is that general pointed Hopf algebras are not generated by
group-likes and skew-primitive elements. While in the characteristic
0 case, the well-known Andruskiewitsch-Schneider Conjecture
\cite{AS3} claims that all finite-dimensional pointed Hopf algebras
are indeed generated by their group-likes and skew-primitive
elements.

This paper can be considered as our first try to apply quiver
methods to the category of pointed Hopf algebras over an
algebraically closed field $k$ with characteristic $p>0$, especially
to finite-dimensional cocommutative Hopf algebras over $k$ or
equivalently to the category of finite algebraic $k$-groups. One can
show that the connected component of a finite-dimensional
cocommutative Hopf algebra can be embedded into the path coalgebra
of a multi-loop quiver (see Corollary 2.2). So as the first step,
one can analyze the minimal case, that is, Hopf structures over the
path coalgebra of the one-loop quiver. This is exactly what we do in
this paper. The main result of the paper is the following theorem
(see Theorem 5.1).
\begin{theorem} Let $n\in \mathbb{N}^{+}$, any commutative Hopf structure $H$ over $k\circlearrowleft_{p^{n}}$ is
isomorphic to a $L(n,d)$ for some $d$.
\end{theorem}
See Section 5 for the definition of $L(n,d)$. As a direct
consequence of this theorem,  all commutative infinitesimal groups
$\mathcal{G}$ with dim$_{k}$Lie$(\mathcal{G})=1$ are determined.

The paper is organized as follows. All needed knowledge about path
coalgebras is summarized in Section 2. Moreover, the uniserial
property of the Hopf structures over $k\circlearrowleft$ is also
established in this section. For later use, all endomorphisms of the
path coalgebra $k\circlearrowleft$ are given in Section 3. As a key
step, we need to grasp all possible Hopf structures over
$k\circlearrowleft_{p}$ at first and this task is finished in
Section 4. In addition, we also show that the property of a Hopf
structure over $k\circlearrowleft_{p^{n}}$ is almost determined by
that of its restriction to $k\circlearrowleft_{p}$. Combining the
work of Farnsteiner-R\"ohrle-Voigt on unipotent group of complexity
1 \cite{FRV}, the proof of Theorem 1.1 is given in Section 5 at
last.

Throughout the paper we will be working over an algebraically closed
field $k$ of characteristic $p>0$. We freely use the results,
notations, and conventions of \cite{Mon}.

\section{Path coalgebras}

\subsection{} Given a quiver  $Q=(Q_{0},Q_{1})$ with $Q_{0}$ the set of
vertices and $Q_{1}$  the set of arrows, denote by $kQ$  the
 $k$-space with basis the set of all paths in $Q$. Over $kQ$, there is a natural coalgebra structure defined as follows.
  For
 $\alpha \in Q_{1}$, let $s(\alpha)$ and $t(\alpha)$ denote
 respectively the starting and ending vertex of $\alpha$.
Then comultiplication $\Delta$ is given by
 $$\Delta(p)=\alpha_{l}\cdots \alpha_{1}
 \otimes s(\alpha_{1})+\sum_{i=1}^{l-1}\alpha_{l}\cdots \alpha_{i+1}
 \otimes\alpha_{i}\cdots \alpha_{1}+ t(\alpha_{l})\otimes \alpha_{l}\cdots \alpha_{1}$$
 for each path $p=\alpha_{l}\cdots \alpha_{1}$ with each $\alpha_{i}\in
 Q_{1}$; and the counit $\varepsilon$ is defined to be  $\varepsilon(p)=0$ for $l\geq 1$ and $1$ if $l=0$
($l=0$ means $p$ is a vertex). This is a coradically graded pointed
coalgebra and we also denote it by $kQ$. Like the path algebras
case, the path coalgebras serve as the cofree pointed coalgebras. In
fact, Chin and Montgomery showed the following result \cite{CM}:

\begin{lemma} Let $C$ be a pointed cocalgebra, then there
exists a unique quiver $Q(C)$ such that $C$ can be embedded into the
path coalgebra $kQ(C)$ as a large sub coalgebra.\end{lemma}

This unique quiver $Q(C)$ is called the \emph{dual Gabriel quiver}
of $C.$ Here ``large" means that $C$ contains all group-like
elements $Q(C)_{0}$ and all skew-primitive elements of $kQ(C)$. Note
that the skew-primitive elements are indeed corresponding to paths
of length 1, i.e., arrows. Now the following conclusion is clear.

\begin{corollary}
Let $C$ be an irreducible cocommutative pointed coalgebra, then its
dual Gabriel quiver $Q(C)$ has only one vertex.
\end{corollary}

 A natural question is when there is a Hopf structure on a path
coalgebra. We will see not every quiver can serve as the dual
Gabriel quiver of a pointed Hopf algebra and those do are called
\emph{Hopf quivers} by Cibils and Rosso \cite{CR}. Recall that a
\emph{ramification data} $r$ of a group $G$ is a positive central
element of the group ring of $G$: let $\mathcal{C}$ be the set of
conjugacy classes, $r =\Sigma _{C\in \;\mathcal{C}} r_{C}C$ is a
formal sum with non-negative integer coefficients.

\begin{definition} Let $G$ be a group and $r$ a
ramification data. The corresponding Hopf quiver $Q(G,r)$ has set of
vertices the elements of $G$ and has $r_{C}$ arrows from $x$ to $cx$
for each $x \in G$ and $c\in C$.
\end{definition}

One of the main results in \cite{CR} states that there is a graded
Hopf algebra structure on the path coalgebra $kQ$ if and only if $Q$
is a Hopf quiver. In this case, $kQ_{0}$ is a group algebra and
$kQ_{1}$ is a $kQ_{0}$-Hopf bimodule. Moreover, the product rule of
paths can be displayed as follows.

Let $p$ be a path of length $l$. An $n$-thin split of it is a
sequence $(p_1, \ \cdots, \ p_n)$ of vertices and arrows such that
the concatenation $p_n \cdots p_1$ is exactly $p.$ These $n$-thin
splits are in one-to-one correspondence with the $n$-sequences of
$(n-l)$ 0's and $l$ 1's. Denote the set of such sequences by
$D_l^n.$ Clearly $|D_l^n|={n \choose l}.$ For $d=(d_1, \ \cdots, \
d_n) \in D_l^n,$ the corresponding $n$-thin split is written as
$dp=((dp)_1, \ \cdots, \ (dp)_n),$ in which $(dp)_i$ is a vertex if
$d_i=0$ and an arrow if $d_i=1.$

Let $\alpha=a_m \cdots a_1$ and $\beta=b_n \cdots b_1$ be paths of
lengths $m$ and $n$ respectively. Let $d \in D_m^{m+n}$ and $\bar{d}
\in D_n^{m+n}$ the complement sequence which is obtained from $d$ by
replacing each 0 by 1 and each 1 by 0. Define an element
 in $kQ_{m+n},$
$$(\alpha \cdot \beta)_d=[(d\alpha)_{m+n}.(\bar{d}\beta)_{m+n}] \cdots
[(d\alpha)_1.(\bar{d}\beta)_1],$$ where
$[(d\alpha)_i.(\bar{d}\beta)_i]$ is understood as the action of
$kQ_0$-Hopf bimodule on $kQ_1$ and these terms in different brackets
are put together by cotensor product, or equivalently concatenation.
In these notations, the formula of the product of $\alpha$ and
$\beta$ is given as follows (see pages 245-246 in \cite{CR}):
\begin{equation}
\alpha \cdot \beta=\sum_{d \in D_m^{m+n}}(\alpha \cdot \beta)_d \ .
\end{equation}

\subsection{} In this paper, we only consider the very simple Hopf quiver,
a loop $\circlearrowleft$. By setting $G:=e$ and $r:=e$, one can see
that a loop is just the Hopf quiver $Q(G,r)$. For any natural number
$n$, denote the unique path of length $n$ by $\alpha_{n}$. Since the
group $G$ is trivial now, the Hopf bimodule action is trivial too.
Thus the product rule over $k\circlearrowleft$ is very simple. That
is,
\begin{equation}
\alpha_{n} \cdot \alpha_{m}=\left (
\begin{array}{cc} m+n \\n
\end{array}\right)\alpha_{m+n}.
\end{equation}
This is indeed the familiar Hopf algebra $(k[x])^{\circ}$, the
finite dual of $k[x]$. Sometimes, we denote this Hopf structure
still by $k\circlearrowleft$ and one can discriminate the exact
meaning by context. Note that this is a graded Hopf algebra with
length grading.

For a quiver $Q$, define $kQ_{d}:=\oplus_{i=0}^{d-1}kQ(i)$ where
$Q(i)$ is the set of all paths of length $i$ in $Q$. Clearly, for
any $i\geq0$, $k\circlearrowleft_{p^{i}}$ is a sub Hopf algebra of
$k\circlearrowleft$.

\begin{lemma} Let $H$ be a
finite-dimensional sub Hopf algebra of $k\circlearrowleft$, then
$H\cong k\circlearrowleft_{p^{i}}$ for some $i\geq 0$.
\end{lemma}
\begin{proof} This is follows directly from the known fact that
$k[x]/(x^{p^{i}})$ are all Hopf quotients of $k[x]$.
\end{proof}

Van Oystaeyen and Zhang proved the dual Gabriel Theorem for
coradically graded pointed Hopf algebras (Theorem 4.5 in \cite{FP}):

\begin{lemma} Let $H$ be a coradically graded pointed Hopf algebra, then its dual Gabriel quiver $Q(H)$ is a
Hopf quiver and there is a Hopf embedding
$$H\hookrightarrow kQ(H).$$
\end{lemma}

Now let $C\subset k\circlearrowleft$ be a finite-dimensional large
sub coalgebra of $k\circlearrowleft$ and assume there is a Hopf
structure $H(C)$ on $C$.

\begin{proposition}
With notations and the assumption as above, there is a natural
number $i$ such that as a coalgebra, $$C\cong
k\circlearrowleft_{p^{i}}.$$
\end{proposition}
\begin{proof} At first, we know that $H(C)$ is a pointed Hopf
algebra. Denote its coradical filtration by $\{ H(C)_n
\}_{n=0}^{\infty}.$ Define \[\operatorname{gr}(H(C))=H(C)_0 \oplus
H(C)_1/H(C)_0 \oplus H(C)_2/H(C)_1 \oplus \cdots \cdots \] as the
corresponding coradically graded version. Then
$\operatorname{gr}(H(C))$ inherits from $H(C)$ a coradically graded
Hopf algebra structure (see e.g. \cite{Mon}). By Lemma 2.5,
$\gr(H(C))$ is a sub Hopf algebra of $k\circlearrowleft$. Thus Lemma
2.4 implies what we want.
\end{proof}

Thus, our next aim is to give all possible Hopf structures (not
necessarily coradically graded) over the coalgebra
$k\circlearrowleft_{p^{i}}$.

For any rational number $a$, denote by $[a]$ the biggest integer
which is not bigger than $a$.
\begin{lemma} For any positive integers $m,n$, $\left (
\begin{array}{cc} m+n \\n
\end{array}\right)=0$ if and only if
$$\sum_{i\geq 1}[\frac{m+n}{p^{i}}]> \sum_{i\geq 1}[\frac{m}{p^{i}}]+ \sum_{i\geq 1}[\frac{n}{p^{i}}].$$
\end{lemma}
\begin{proof} Clear.
\end{proof}

 We call a Hopf algebra is
\emph{uniserial} if the set of its sub Hopf algebras forms a totally
ordered set under the containing relation.

 \textbf{Convention. } Let $C$ and $D$ be two coalgebras and assume that $C$ is a sub coalgebra of $D$.
  If there is a Hopf structure $H(D)$ over
$D$, then we use the notion $H(C)$ to denote the restriction, if
applicable, of the structure of $H(D)$ to $C$.

\begin{proposition} Let $n$ be a positive natural number and assume
that there is a Hopf structure $H(k\circlearrowleft_{p^{n}})$ over
$k\circlearrowleft_{p^{n}}$. Then $H(k\circlearrowleft_{p^{n}})$ is
a uniserial Hopf algebra with the composition series
$$k\subset H(k\circlearrowleft_{p^{1}})\subset \cdots\subset  H(k\circlearrowleft_{p^{i-1}})\subset
H(k\circlearrowleft_{p^{i}})\subset \cdots\subset
H(k\circlearrowleft_{p^{n}}).$$
\end{proposition}
\begin{proof} By Proposition 2.6, it is enough to show that
$H(k\circlearrowleft_{p^{i}})$ for $i\leq n$ are sub Hopf algebras.
Thus, it is enough to show that they are closed under the
multiplication. But this is the direct consequence of the product
rule (2.1) and Lemma 2.7.
\end{proof}

\section{endomorphisms of $k\circlearrowleft$}

For later use, we characterize all the possible endomorphisms of the
path coalgebra $k\circlearrowleft$ in this section.

\begin{theorem}
\emph{(i)} Let $f:\;k\circlearrowleft\rightarrow k\circlearrowleft$
be a coalgebra map, then there are $\{\lambda_{i}\in k|i\in
\mathbb{N}^{+}\}$ such that
$$f(\alpha_{n})=\sum_{r=1}^{n}(\sum_{n_{1}+\cdots n_{r}=n}\lambda_{n_{1}}\cdots\lambda_{n_{r}})\alpha_{r}$$
for any $n$.

\emph{(ii)} All coalgebra endomorphisms of $k\circlearrowleft$ are
precisely given in this way.
\end{theorem}
\begin{proof} (i) Let's find such $\lambda_{i}'$s. Since $f$ is a
coalgebra map, $f(1)$ is a group-like element and $f(\alpha_{1})$ is
a primitive element. Thus $f(1)=1$ and there is a $\lambda_{1}\in k$
such that $f(\alpha_{1})=\lambda_{1}\alpha_{1}$ since $k\alpha_{1}$
are all primitive elements. Suppose we have found
$\{\lambda_{1},\ldots,\lambda_{n}\}$ and let's find $\lambda_{n+1}$.
By $f$ is a coalgebra map,
$$\Delta(f(\alpha_{n+1})-\sum_{r=2}^{n+1}(\sum_{n_{1}+\cdots
n_{r}=n+1}\lambda_{n_{1}}\cdots\lambda_{n_{r}})\alpha_{r})\;\;\;\;\;\;\;\;\;\;\;\;\;\;\;\;\;\;
\;\;\;\;\;\;\;\;\;\;\;\;\;\;\;\;\;\;\;\;\;\;\;\;\;\;\;\;\;\;\;\;\;\;\;\;$$
\begin{eqnarray*}
&=&f(\alpha_{n+1})\otimes 1+ 1\otimes
f(\alpha_{n+1})+\sum_{i=1}^{n}f(\alpha_{i})\otimes
f(\alpha_{n+1-i})-\\
&\;&\sum_{r=2}^{n+1}(\sum_{n_{1}+\cdots
n_{r}=n+1}\lambda_{n_{1}}\cdots\lambda_{n_{r}})\sum_{s+t=r}\alpha_{s}\otimes \alpha_{t})\\
&=&f(\alpha_{n+1})\otimes 1+ 1\otimes
f(\alpha_{n+1})+\sum_{i=1}^{n}(\sum_{s=1}^{i}(\sum_{n_{1}+\cdots
n_{s}=i}\lambda_{n_{1}}\cdots\lambda_{n_{s}})\alpha_{s}\\
&\;&\otimes \sum_{t=1}^{n+1-i}(\sum_{m_{1}+\cdots
m_{t}=n+1-i}\lambda_{m_{1}}\cdots\lambda_{m_{t}})\alpha_{t})-\\
&\;&\sum_{r=2}^{n+1}(\sum_{n_{1}+\cdots
n_{r}=n+1}\lambda_{n_{1}}\cdots\lambda_{n_{r}})\sum_{s+t=r}\alpha_{s}\otimes
\alpha_{t}).
\end{eqnarray*}

Replace $\lambda_{m_{1}}\cdots\lambda_{m_{t}}$ by
$\lambda_{n_{s+1}}\cdots\lambda_{n_{s+t}}$ and set $r=s+t$, one can
find that
$$ \sum_{i=1}^{n}(\sum_{s=1}^{i}(\sum_{n_{1}+\cdots
n_{s}=i}\lambda_{n_{1}}\cdots\lambda_{n_{s}})\alpha_{s}\otimes
\sum_{t=1}^{n+1-i}(\sum_{m_{1}+\cdots
m_{t}=n+1-i}\lambda_{m_{1}}\cdots\lambda_{m_{t}})\alpha_{t})$$
$$=\sum_{r=2}^{n+1}(\sum_{n_{1}+\cdots
n_{r}=n+1}\lambda_{n_{1}}\cdots\lambda_{n_{r}})\sum_{s+t=r,s\neq
0\neq t}\alpha_{s}\otimes
\alpha_{t}).\;\;\;\;\;\;\;\;\;\;\;\;\;\;\;\;\;\;\;\;\;\;\;\;\;\;\;\;\;\;\;\;(\star)$$
Thus let $y:=f(\alpha_{n+1})-\sum_{r=2}^{n+1}(\sum_{n_{1}+\cdots
n_{r}=n+1}\lambda_{n_{1}}\cdots\lambda_{n_{r}})\alpha_{r}$, then
$\Delta(y)=y\otimes 1+1\otimes y$. Thus
$f(\alpha_{n+1})-\sum_{r=2}^{n+1}(\sum_{n_{1}+\cdots
n_{r}=n+1}\lambda_{n_{1}}\cdots\lambda_{n_{r}})\alpha_{r}$ is a
primitive element and so there is a $\lambda_{n+1}$ such that
$y=\lambda_{n+1}\alpha_{1}$. Equivalently,
$$f(\alpha_{n+1})=\sum_{r=1}^{n+1}(\sum_{n_{1}+\cdots n_{r}=n+1}\lambda_{n_{1}}\cdots\lambda_{n_{r}})\alpha_{r}.$$

(ii) By (i), it is enough to show that for any $\{\lambda_{i}\in
k|i\in \mathbb{N}^{+}\}$ and the linear map $f$ defined by
$f(\alpha_{n})=\sum_{r=1}^{n}(\sum_{n_{1}+\cdots
n_{r}=n}\lambda_{n_{1}}\cdots\lambda_{n_{r}})\alpha_{r}$ is indeed a
coalgebra map. In fact,
\begin{eqnarray*} (f\otimes f)\Delta(\alpha_{n})&=&f(\alpha_{n})\otimes 1+ 1\otimes
f(\alpha_{n})+\sum_{i=1}^{n-1}f(\alpha_{i})\otimes f(\alpha_{n-i})\\
&=&f(\alpha_{n})\otimes 1+ 1\otimes
f(\alpha_{n})+\sum_{i=1}^{n-1}(\sum_{s=1}^{i}(\sum_{n_{1}+\cdots
n_{s}=i}\lambda_{n_{1}}\cdots\lambda_{n_{s}})\alpha_{s}\\
&\;&\otimes \sum_{t=1}^{n-i}(\sum_{m_{1}+\cdots
m_{t}=n-i}\lambda_{m_{1}}\cdots\lambda_{m_{t}})\alpha_{t}).
\end{eqnarray*}
While
\begin{eqnarray*} \Delta(f(\alpha_{n}))&=&\Delta(\sum_{r=1}^{n}(\sum_{n_{1}+\cdots
n_{r}=n}\lambda_{n_{1}}\cdots\lambda_{n_{r}})\alpha_{r})\\
&=&\sum_{r=1}^{n}(\sum_{n_{1}+\cdots
n_{r}=n}\lambda_{n_{1}}\cdots\lambda_{n_{r}})\sum_{s+t=r}(\alpha_{s}\otimes
\alpha_{t})\\
&=&f(\alpha_{n})\otimes 1+1\otimes f(\alpha_{n})+
\sum_{r=2}^{n}(\sum_{n_{1}+\cdots
n_{r}=n}\lambda_{n_{1}}\cdots\lambda_{n_{r}})\\
&\;\;\;\;&\times \sum_{s+t=r,s\neq 0\neq t}(\alpha_{s}\otimes
\alpha_{t}).
\end{eqnarray*}
Then equation $(\star)$ implies that
$\Delta(f(\alpha_{n}))=(f\otimes f)\Delta(\alpha_{n})$.
\end{proof}

By the proof, we know that if $f:\;k\circlearrowleft\rightarrow
k\circlearrowleft$ is a coalgebra map, then
$f(\alpha_{1})=\lambda_{1}\alpha_{1}$ for some $\lambda_{1}\in k$.
The next result is to provide a criterion to determine when $f$ is
indeed an automorphism.

\begin{proposition} With notions as the above, $f$ is an
automorphism if and only if $\lambda_{1}\neq 0$.
\end{proposition}
\begin{proof} By Theorem 3.1, $f(k\circlearrowleft_{n})\subseteq
k\circlearrowleft_{n}$ for any $n\in \mathbb{N}^{+}$ and thus
$f|_{k\circlearrowleft_{n}}$ is a coalgebra endomorphism of
$k\circlearrowleft_{n}$. By $\lambda_{1}\neq 0$,
$f|_{k\circlearrowleft_{2}}$ is injective and so
$f|_{k\circlearrowleft_{n}}$ is injective by Heynaman-Radford's
result \cite{HR}. Since dim$_{k}k\circlearrowleft_{n}<\infty$,
$f|_{k\circlearrowleft_{n}}$ is bijective. This indeed implies that
$f$ is an automorphism of $k\circlearrowleft$. The converse is
obvious since one always has $f(\alpha_{1})=\lambda_{1}\alpha_{1}.$
\end{proof}

\begin{corollary} For any natural numbers $m>n>0$ and assume that $f$ is an automorphism
of the coalgebra $k\circlearrowleft_{n}$, then $f$ can be extended
to be automorphisms of the coalgebra $k\circlearrowleft$ and
$k\circlearrowleft_{m}$.
\end{corollary}
\begin{proof} By the proof of Theorem 3.1, there are
$\{\lambda_{1},\ldots,\lambda_{n-1}\}$ such that
$$f(\alpha_{i})=\sum_{r=1}^{i}(\sum_{n_{1}+\cdots
n_{r}=n}\lambda_{n_{1}}\cdots\lambda_{n_{r}})\alpha_{r}$$ for $1\leq
i\leq n-1$. By setting $\lambda_{j}=0$ for all $j\geq n$ and by
Theorem 3.1, if we define a map $F:\;k\circlearrowleft\rightarrow
k\circlearrowleft$ through
$$F(\alpha_{l})=\sum_{r=1}^{l}(\sum_{n_{1}+\cdots n_{r}=l}\lambda_{n_{1}}\cdots\lambda_{n_{r}})\alpha_{r}$$
for any natural number $l$, then $F$ is a coalgebra endomorphism of
$k\circlearrowleft$. Clearly, $F|_{k\circlearrowleft_{n}}=f$. Owning
to Proposition 3.2, $F$ is an automorphism. Theorem 3.1 deduces that
$F({k\circlearrowleft_{m}})\subseteq k\circlearrowleft_{m}$ and thus
$F|_{k\circlearrowleft_{m}}$ is the extension of $f$ to
$k\circlearrowleft_{m}$.
\end{proof}

\section{Hopf structures on $k\circlearrowleft_{p}$}

The following result seems well-known and we write its proof out for
completeness.

\begin{lemma} Let $H$ be a Hopf structure over
$k\circlearrowleft_{p^{1}}$, then as a Hopf algebra $H$ is
isomorphic to either $(k\mathbb{Z}_{p})^{\ast}$, dual of the group
algebra $k\mathbb{Z}_{p}$, or $k[x]/(x^{p})$.
\end{lemma}
\begin{proof} At first, it is not hard to see that $H$ is generated
by $\alpha_{p^{0}}=\alpha_{1}$. Consider the element
$\alpha_{1}^{p}$. For it, we have
$$\Delta(\alpha_{1}^{p})=\Delta(\alpha_{1})^{p}=(1\otimes \alpha_{1}+ \alpha_{1}\otimes 1)^{p}=1\otimes \alpha_{1}^{p}
+\alpha_{1}^{p}\otimes 1. $$ Thus $\alpha_{1}^{p}$ is a primitive
element. Since the space spanned by $\alpha_{1}$ are all primitive
elements in the coalgebra $k\circlearrowleft_{p^{1}}$, there is a
$\lambda\in k$ such that
$$\alpha_{1}^{p}=\lambda \alpha_{1}.$$
If $\lambda=0$, then $H\cong k[x]/(x^{p})$. If $\lambda\neq 0$, take
$\lambda'$ to be a solution of the equation $\lambda x^{p}-x=0$.
Then
$$(\lambda'\alpha_{1})^{p}=\lambda\lambda'^{p}\alpha_{1}^{p}=\lambda'\alpha_{1}.$$
In one word, if $\lambda\neq 0$, we can always assume that
$\lambda=1$ and thus $H\cong (k\mathbb{Z}_{p})^{\ast}$.
\end{proof}

We find that the property of $H(k\circlearrowleft_{p^{n}})$ is
largely determined by that of $H(k\circlearrowleft_{p^{1}})$.

\begin{proposition}  Let $n$ be a positive integer and assume
that there is a commutative Hopf structure
$H(k\circlearrowleft_{p^{n}})$ over $k\circlearrowleft_{p^{n}}$. If
$H(k\circlearrowleft_{p^{1}})\cong (k\mathbb{Z}_{p})^{\ast}$, then
$$H(k\circlearrowleft_{p^{n}})\cong (k\mathbb{Z}_{p^{n}})^{\ast}.$$
\end{proposition}
\begin{proof} \textbf{Claim. } \emph{Up to a Hopf isomorphism, $\alpha_{l}^{p}=\alpha_{l}$ for $0<l<p^{n}$. }
We prove this fact by using induction on $l$. If $l=1$, this is just
assumption. Assume that $\alpha_{l}^{p}=\alpha_{l}$ for $l\leq m-1$,
let's prove that $\alpha_{m}^{p}=\alpha_{{m}}.$ By the definition of
path coalgebra and the assumption of commutativity, we always have
\begin{eqnarray*}\Delta(\alpha_{{m}}^{p})&=&(1\otimes \alpha_{{m}}+\alpha_{{m}}\otimes 1+\sum_{0<l<{m}}\alpha_{l}\otimes
\alpha_{{m}-l})^{p}\\
&=&1\otimes \alpha_{{m}}^{p}+\alpha_{{m}}^{p}\otimes
1+\sum_{0<l<{m}}\alpha_{l}^{p}\otimes \alpha_{{m}-l}^{p}.
\end{eqnarray*}
The inductive assumption implies that $\alpha_{l}^{p}=\alpha_{l}$
for $l<{m}$. Thus
$$\Delta(\alpha_{{m}}^{p})=1\otimes \alpha_{{m}}^{p}+\alpha_{{m}}^{p}\otimes
1+\sum_{0<l<{m}}\alpha_{l}\otimes \alpha_{{m}-l}$$ and so
$$\Delta(\alpha_{{m}}^{p}-\alpha_{{m}})=(\alpha_{{m}}^{p}-\alpha_{{m}})\otimes 1+1\otimes (\alpha_{{m}}^{p}-\alpha_{{m}}).$$
Therefore, there is $\lambda\in k$ such that
$\alpha_{{m}}^{p}-\alpha_{{m}}=\lambda\alpha_{1}$. If $\lambda=0$,
done. If $\lambda\neq 0$, take $\xi$ to be a solution of the
equation $x^{p}-x+\lambda=0$ and let
$\alpha_{{m}}':=\alpha_{{m}}+\xi\alpha_{1}$. Clearly, the map
$$f:\;k\circlearrowleft_{m+1}\rightarrow k\circlearrowleft_{m+1},\;\;\alpha_{i}\mapsto \alpha_{i}\;\;\textrm{for}\;i\neq
m;\;\;\alpha_{m}\mapsto\alpha_{m}'
$$ is an automorphism of $k\circlearrowleft_{m+1}$. Corollary 3.3 implies $f$ can be extended
to be an automorphism of $k \circlearrowleft_{p^{n}}$. Since this
automorphism is equivalent to choose a new basis of
$k\circlearrowleft_{p^{n}}$, $f$ is an automorphism of Hopf algebras
of $H(k\circlearrowleft_{p^{n}})$. Now
$$(\alpha_{p^{m}}')^{p}=(\alpha_{p^{m}}+\xi\alpha_{1})^{p}=\alpha_{p^{m}}+\lambda\alpha_{1}+\xi^{p}\alpha_{1}
=\alpha_{p^{m}}+\xi\alpha_{1}=\alpha_{p^{m}}'.$$ The claim is
proved.

 Construct the element
$$t:=\prod_{0<l<p^{n}}(1-\alpha_{l}^{p-1}).$$
Thus the claim implies for any $0<m<p^{n}$,
$$\alpha_{m}t=0=\varepsilon(\alpha_{m})t.$$ This means that $t\in
\int_{H}$, the set of integrals. Since $\varepsilon(t)=1\neq 0$,
$H(k\circlearrowleft_{p^{n}})$ is a simisimple Hopf algebra (Theorem
2.2.1 in \cite{Mon}). Thus $H(k\circlearrowleft_{p^{n}})\cong
(kG)^{\ast}$ for some finite abelian group. Since
$H(k\circlearrowleft_{p^{n}})$ is cogenerated by $\alpha_{1}$, $kG$
is generated by one element. Thus $G\cong \mathbb{Z}_{p^{n}}$.
\end{proof}

\begin{remark} We would like to thank Professor A. Masuoka for
pointing out to us that the above proposition can be deduced from
Chapter IV, Section 3, 3.4 of \cite{DG} or Theorem 0.1 in
\cite{Mas}.
\end{remark}

Recall an affine algebraic group $\mathcal{G}$ is \emph{finite} if
its coordinate ring $\mathcal{O(G)}$ is a finite-dimensional Hopf
algebra. A finite algebraic group $\mathcal{G}$ is called
\emph{infinitesimal} if $\mathcal{O(G)}$ is a local algebra. And, we
call a finite algebraic group $\mathcal{G}$ \emph{unipotent} if its
distribution algebra $\mathcal{H(G)}:=(\mathcal{O(G)})^{\ast}$ is a
local algebra. There is an equivalence between the category of
finite algebraic groups and the category of finite-dimensional
cocommutative Hopf algebras. Explicitly, sending finite algebraic
group $\mathcal{G}$ to $\mathcal{H(G)}$ gives us the equivalence.
For more knowledge about affine algebraic groups, see \cite{DG,Wat}.

Now assume that there is a commutative Hopf structure
$H(k\circlearrowleft_{p^{n}})$ over $k\circlearrowleft_{p^{n}}$,
then there is a finite algebraic group $\mathcal{G}_{p^{n}}$ such
that
$$\mathcal{H}(\mathcal{G}_{p^{n}})=H(k\circlearrowleft_{p^{n}}).$$

\begin{proposition} Keep the above notations. If $H(k\circlearrowleft_{p})\cong
k[x]/(x^{p})$, then $\mathcal{G}_{p^{n}}$ is an infinitesimal
unipotent group.
\end{proposition}
\begin{proof} Owning to the fact that $(k\circlearrowleft_{p^{n}})^{\ast}\cong
k[x]/(x^{p^{n}})$ as algebras, $\mathcal{G}_{p^{n}}$ is
infinitesimal. So in order to show $\mathcal{G}_{p^{n}}$ is
unipotent, it is enough to show that $H(k\circlearrowleft_{p^{n}})$
is a local algebra.

By Proposition 2.8, $H(k\circlearrowleft_{p^{n}})$ is uniserial with
the composition series $$k\subset
H(k\circlearrowleft_{p^{1}})\subset \cdots\subset
H(k\circlearrowleft_{p^{i-1}})\subset
H(k\circlearrowleft_{p^{i}})\subset \cdots\subset
H(k\circlearrowleft_{p^{n}}).$$  Lemma 4.1 implies that either
$H(k\circlearrowleft_{p^{i}})/H(k\circlearrowleft_{p^{i-1}})^{+}H(k\circlearrowleft_{p^{i}})\cong
(k\mathbb{Z}_{p})^{\ast}$ or $H(k\circlearrowleft_{p^{i}})
/H(k\circlearrowleft_{p^{i-1}})^{+}H(k\circlearrowleft_{p^{i}})$
$\cong k[x]/(x^{p})$ for any $1\leq i\leq n$. Here for a Hopf
algebra $H$, $H^{+}$ stands for the kernel of $\varepsilon:
\;H\rightarrow k$. To show $H(k\circlearrowleft_{p^{n}})$ is local,
it is enough to show
$H(k\circlearrowleft_{p^{i}})/H(k\circlearrowleft_{p^{i-1}})^{+}H(k\circlearrowleft_{p^{i}})\cong
k[x]/(x^{p})$ for all $1\leq i\leq n$ (In fact, if so then all
non-trivial paths will be nilpotent).

Otherwise, there is an $i$ such that
$H(k\circlearrowleft_{p^{i}})/H(k\circlearrowleft_{p^{i-1}})^{+}H(k\circlearrowleft_{p^{i}})\cong
(k\mathbb{Z}_{p})^{\ast}$. Take such $i$ as small as possible. By
assumption, $i\geq 2$. Thus
$H(k\circlearrowleft_{p^{i}})/H(k\circlearrowleft_{p^{i-1}})^{+}H(k\circlearrowleft_{p^{i}})\cong
(k\mathbb{Z}_{p})^{\ast}$ implies that $$\alpha_{p^{i-1}}^{p}\equiv
\alpha_{p^{i-1}}\;\;\textrm{mod}\;k\circlearrowleft_{p^{i-1}}.$$ And
thus $\alpha_{p^{i-1}}^{p^{i}}\equiv
\alpha_{p^{i-1}}\;\;\textrm{mod}\;k\circlearrowleft_{p^{i-1}}.$
Therefore there is an element $a\in k\circlearrowleft_{p^{i-1}}$
such that $\alpha_{p^{i-1}}^{p^{i-1}}=\alpha_{p^{i-1}}+a$. Since $i$
is as small as possible, $H(k\circlearrowleft_{p^{i-1}})$ is local
and all non-trivial paths in $k\circlearrowleft_{p^{i-1}}$ are
nilpotent. More precisely, let $\alpha$ be a non-trivial path living
in $k\circlearrowleft_{p^{i-1}}$, then $\alpha^{p^{i-1}}=0$. Thus
\begin{eqnarray*} \Delta(\alpha_{p^{i-1}}^{p^{i-1}})&=&(1\otimes \alpha_{p^{i-1}}+\alpha_{p^{i-1}}\otimes 1
+\sum_{0<l<p^{i-1}}\alpha_{l}\otimes \alpha_{p^{i-1}-l})^{p^{i-1}}\\
&=&1\otimes
\alpha_{p^{i-1}}^{p^{i-1}}+\alpha_{p^{i-1}}^{p^{i-1}}\otimes
1+\sum_{0<l<p^{i-1}}\alpha_{l}^{p^{i-1}}\otimes
\alpha_{p^{i-1}-l}^{p^{i-1}}\\
&=& 1\otimes
\alpha_{p^{i-1}}^{p^{i-1}}+\alpha_{p^{i-1}}^{p^{i-1}}\otimes 1.
\end{eqnarray*}
This implies that $\alpha_{p^{i-1}}^{p^{i-1}}=\alpha_{p^{i-1}}+a$ is
a primitive element and so there is a $\lambda\in k$ such that
$\alpha_{p^{i-1}}+a=\lambda\alpha_{1}$. Therefore,
$\alpha_{p^{i-1}}\in k\circlearrowleft_{p^{i-1}}$ which is
impossible.
\end{proof}

Combining Lemma 4.1, Propositions 4.2, 4.4 and 2.8, we get

\begin{corollary}  Let $n$ be a positive integer and assume
that there is a commutative Hopf structure
$H(k\circlearrowleft_{p^{n}})$ over $k\circlearrowleft_{p^{n}}$.
Then either $H(k\circlearrowleft_{p^{n}})\cong
(k\mathbb{Z}_{p^{n}})^{\ast}$ or $H(k\circlearrowleft_{p^{n}})$ is
the distribution algebra of a uniserial infinitesimal unipotent
commutative $k$-group.
\end{corollary}

\section{Classification and application}

 Fix a positive integer $n$ and consider the coalgebra
 $k\circlearrowleft_{p^{n}}$. Assume that there is a Hopf structure on
 $k\circlearrowleft_{p^{n}}$. Since its coradically graded version
 is generated by $\{\alpha_{p^{i}}|1\leq i\leq n-1\}$, it is also
 generated by $\{\alpha_{p^{i}}|1\leq i\leq n-1\}$. So in order to
 give the Hopf structure, it is enough to characterize the
 relations between $\{\alpha_{p^{i}}|1\leq i\leq n-1\}$.

  For any $0\leq d\leq n$, the Hopf algebra $L(n,d)$ (it is indeed a Hopf algebra by the following theorem) is defined to
  be the Hopf algebra over $k\circlearrowleft_{p^{n}}$ with
  relations:
  \begin{equation} \alpha_{p^{i}}\alpha_{p^{j}}=\alpha_{p^{j}}\alpha_{p^{i}},\;\;\textrm{for}\;0\leq i,j\leq
  n-1;\end{equation}
  \begin{equation} \alpha_{p^{i}}^{p}=0,\;\;\textrm{for}\;i< d;
  \end{equation}
   \begin{equation} \alpha_{p^{i}}^{p}=\alpha_{p^{i-d}},\;\;\textrm{for}\;i\geq
   d.
  \end{equation}

The main result of this section is the following.

\begin{theorem} $L(n,d)$ is a Hopf algebra and any commutative
 Hopf structure $H(k\circlearrowleft_{p^{n}})$ over $k\circlearrowleft_{p^{n}}$ is isomorphic
to an $L(n,d)$ for some $d$.
\end{theorem}

One of the main ingredients of the proof is the classification
result given in \cite{FRV}. Let's recall it. By
$\mathcal{W}:\;\mathbb{M}_{k}\rightarrow \mathbb{M}_{\mathbb{Z}}$ we
denote the affine commutative group scheme of \emph{Witt vectors}.
For any positive natural number $m$ let
$\mathcal{W}_{m}:\;\mathbb{M}_{k}\rightarrow
\mathbb{M}_{\mathbb{Z}}$ be the affine commutative group scheme of
\emph{Witt vectors of length $m$}. Denote the \emph{Frobenius map}
and \emph{Verschiebung} of $\mathcal{W}_{m}$ by $\mathcal{F}$ and
$\mathcal{V}$ respectively. For any finite commutative algebraic
group $\mathcal{G}$, its \emph{Cartier dual} is denoted by
$\mathcal{D(G)}$. For details, see \cite{DG}. An infinitesimal
unipotent commutative group $\mathcal{U}$ is called
$\mathcal{V}$-\emph{uniserial} if Coker$\mathcal{V}\cong
\Spec_{k}(k[x]/(x^{p}))$. Likewise, a unipotent infinitesimal group
$\mathcal{U}$ is called $\mathcal{F}$-\emph{uniserial} if
Ker$\mathcal{F}\cong\Spec_{k}(k[x]/(x^{p}))$. Note that
$\mathcal{G}$ is $\mathcal{V}$-uniserial or $\mathcal{F}$-uniserial
is equivalent to  its distribution algebra $\mathcal{H(G)}$ is
uniserial (see Lemma 2.5 in \cite{FRV}).

Let $d,j,n\in \mathbb{N}$ and for $n\geq 1,\;d\geq 1$, we denote by
$\mathcal{U}_{n,d}$ the kernel of the endomorphism
$\mathcal{V}^{d}-\mathcal{F}:\; \mathcal{W}_{m}\rightarrow
\mathcal{W}_{m}$ with $m=n(d+1)$. Denote by $\mathcal{U}_{n,d}^{j}$
the intersection of $\mathcal{U}_{n,d}$ with the kernel of the
endomorphism
$\mathcal{V}^{(n-1)(d+1)+j}:\;\mathcal{W}_{m}\rightarrow
\mathcal{W}_{m}$ for $1\leq j\leq d$. The following is the main
result of \cite{FRV} (Theorem 1.2 in \cite{FRV}).

\begin{lemma} The following gives a complete list of representatives of
isomorphism classes of non-trivial uniserial infinitesimal unipotent
commutative $k$-groups:

\emph{(i) }$(\mathcal{W}_{d})_{1}$; $\;\;$\emph{(ii)}
$\mathcal{U}_{n,d}$; $\;\;$ \emph{(iii)}  $\mathcal{U}_{n,d}^{j}$;

\emph{(iv)} $\mathcal{D}((\mathcal{W}_{d})_{1})$; $\;\;$\emph{(v)}
$\mathcal{D}(\mathcal{U}_{n,d})$; $\;\;$ \emph{(vi)}
$\mathcal{D}(\mathcal{U}_{n,d}^{j})$.

Moreover, the groups labeled (i)-(iii) are $\mathcal{V}$-uniserial
and those in (iv)-(vi) are $\mathcal{F}$-uniserial.
\end{lemma}

 \textbf{Proof of the Theorem 5.1. } At first, since $H(k\circlearrowleft_{p^{n}})$
 is commutative, there is a $k$-group $\mathcal{G}$ such that
 $\mathcal{G}=\Spec_{k}(H(k\circlearrowleft_{p^{n}}))$. So the Frobenius map $\mathcal{F}$
 and Verschiebung $\mathcal{V}$ can be defined for
 $H(k\circlearrowleft_{p^{n}})$ too. Let's see what they are.
 In order to explain our understanding, there is no harm to assume that both $\mathcal{F}$ and
 $\mathcal{V}$ are Hopf endomorphisms of $H(k\circlearrowleft_{p^{n}})$ for simplicity
 since the path coalgebra can clearly be defined over $\mathbb{Z}$.
 By the definition of Frobenius map, we know that
 $$\mathcal{F}:\;H(k\circlearrowleft_{p^{n}})\rightarrow H(k\circlearrowleft_{p^{n}}),\;\;
 \alpha_{p^{i}}\mapsto  \alpha_{p^{i}}^{p},\;\;\textrm{for}\;0<i<n.$$
 Note that $\mathcal{V}$ is just the dual map of Frobenius map of
 $\mathcal{D}(\mathcal{G})$. Since as an algebra we have $(k\circlearrowleft_{p^{n}})^{\ast}\cong
 k[x]/(x^{p^{n}})$, the Frobenius map for
 $(k\circlearrowleft_{p^{n}})^{\ast}$ is given by $x\mapsto x^{p}$.
 Note also that $\{\alpha_{i}|0\leq i\leq p^{n}-1\}$ are the dual basis of $\{x^{i}|0\leq i\leq
 p^{n}-1\}$. Thus $\mathcal{V}$ is given by
 $$\mathcal{V}:\;H(k\circlearrowleft_{p^{n}})\rightarrow H(k\circlearrowleft_{p^{n}}),\;\;
 \alpha_{p^{i}}\mapsto  \alpha_{p^{i-1}},\;\;\textrm{for}\;0<i<n.$$
 Thus if $\Spec_{k}(H(k\circlearrowleft_{p^{n}}))$ is unipotent,
 then  it is a $\mathcal{V}$-uniserial group.

 According to Corollary 4.5, either $H(k\circlearrowleft_{p^{n}})\cong
(k\mathbb{Z}_{p^{n}})^{\ast}$ or $H(k\circlearrowleft_{p^{n}})$ is
the distribution algebra of a uniserial infinitesimal unipotent
commutative $k$-group. If $H(k\circlearrowleft_{p^{n}})\cong
(k\mathbb{Z}_{p^{n}})^{\ast}$, then
$H(k\circlearrowleft_{p^{n}})\cong L(n,0)$. Otherwise,
$H(k\circlearrowleft_{p^{n}})$ is a local algebra which implies that
$\Spec_{k}(H(k\circlearrowleft_{p^{n}}))$ is infinitesimal and thus
a unipotent group. By the discussion above,
$\Spec_{k}(H(k\circlearrowleft_{p^{n}}))$ is an infinitesimal
unipotent $\mathcal{V}$-uniserial group. By Lemma 5.2, we have
$\Spec_{k}(H(k\circlearrowleft_{p^{n}}))\cong (\mathcal{W}_{d})_{1}$
or $\Spec_{k}(H(k\circlearrowleft_{p^{n}}))\cong \mathcal{U}_{m,d}$
or $\Spec_{k}(H(k\circlearrowleft_{p^{n}}))\cong
\mathcal{U}_{m,d}^{j}$ for some $m,d,j$. The first case implies that
$H(k\circlearrowleft_{p^{n}})\cong L(n,n)$. Let us analyze the last
two cases. Recall that the coordinate ring of $\mathcal{W}_{n}$ is
$k[x_{1},\ldots,x_{n}]$. If $(d+1)|n$ (that is, we consider the
second case), we have a Hopf epimorphism
$$\pi:\;k[x_{1},\ldots,x_{n}]\twoheadrightarrow H(k\circlearrowleft_{p^{n}})$$
and the following commutative diagram

\begin{figure}[hbt]
\begin{picture}(100,50)(0,20)
\put(0,70){\makebox(0,0){$ \mathcal{O}(\mathcal{W}_{n})$}}
\put(20,70){\vector(1,0){50}} \put(100,70){\makebox(0,0){$
H(k\circlearrowleft_{p^{n}})$}} \put(50,75){\makebox(0,0){$ \pi$}}

\put(10,60){\vector(0,-1){50}}\put(90,60){\vector(0,-1){50}}
\put(-15,35){\makebox(0,0){$ \mathcal{V}^{d}-\mathcal{F}$}}
\put(115,35){\makebox(0,0){$ \mathcal{V}^{d}-\mathcal{F}$}}

\put(0,0){\makebox(0,0){$ \mathcal{O}(\mathcal{W}_{n})$}}
\put(20,0){\vector(1,0){50}} \put(100,0){\makebox(0,0){$
H(k\circlearrowleft_{p^{n}})$}}\put(50,5){\makebox(0,0){$ \pi$}}

\end{picture}
\end{figure}

By $\Spec_{k}(H(k\circlearrowleft_{p^{n}}))\cong
\mathcal{U}_{\frac{n}{d+1},d}$,
dim$_{k}\mathcal{O}(\mathcal{U}_{\frac{n}{d+1},d})=p^{n}$. Therefore
the above commutative diagram and the definitions of
$\mathcal{F},\mathcal{V}$ for $k\circlearrowleft_{p^{n}}$ imply that
$H(k\circlearrowleft_{p^{n}})$ satisfies equations (5.1)-(5.3)
automatically. By comparing the dimension, equations (5.1)-(5.3) are
all the relations for $H(k\circlearrowleft_{p^{n}})$. Thus
$H(k\circlearrowleft_{p^{n}})\cong L(n,d)$ with $(d+1)|n$. For the
last case (that is, $\Spec_{k}(H(k\circlearrowleft_{p^{n}}))\cong
\mathcal{U}_{m,d}^{j}$), the analysis is almost the same as the
second case and the only point we need to say is that the condition
``intersection with kernel of the endomorphism
$\mathcal{V}^{(n-1)(d+1)+j}$", appearing in the definition of
$\mathcal{U}_{m,d}^{j}$, is equivalent to the condition $(d+1)\nmid
n$.

Of course, $L(n,d)$ are all Hopf algebras now. In fact, the above
discussions show that we have $$L(n,0)\cong
(k\mathbb{Z}_{p^{n}})^{\ast},\;\;L(n,n)\cong k[x]/(x^{p^{n}}),$$ and
$L(n,d)\cong \mathcal{O}(\mathcal{U}_{\frac{n}{d+1},d})$ in case
$(d+1)|n$. If $(d+1)\nmid n$, then $n=m(d+1)+j$ for some $m,j$ with
$0<j<d+1$. The above discussions indicate that  $L(n,d)\cong
\mathcal{O}(\mathcal{U}^{j}_{m+1,d})$. $\;\;\;\;\;\;\;\;\square$

\begin{corollary} Up to Hopf
isomorphisms there are exactly $n+1$ classes of non-isomorphic
commutative Hopf structures on the coalgebra
$k\circlearrowleft_{p^{n}}$ for any natural number $n$.
\end{corollary}

As another direct consequence of this theorem, the commutative
infinitesimal groups with 1-dimensional Lie algebras can be
classified now.

\begin{corollary} Let $\mathcal{G}$ be a commutative infinitesimal
group. If dim$_{k}$Lie$(\mathcal{G})=1$, then $\mathcal{H(G)}\cong
L(n,d)$ for some $n,d$.
\end{corollary}
\begin{proof} By dim$_{k}$Lie$(\mathcal{G})=1$, the set of primitive
elements of $\mathcal{H(G)}$ is 1-dimensional. Note that
$\mathcal{H(G)}$ is always pointed, $\mathcal{H(G)}$ can be embedded
into the path coalgebra $k\circlearrowleft$ (Lemma 2.1). Thus there
is a natural number $n$ such that $\mathcal{H(G)}$ is a Hopf
structure over $k\circlearrowleft_{p^{n}}$ by Proposition 2.6.
Thanks to Theorem 5.1, $\mathcal{H(G)}\cong L(n,d)$ for some $d$.
\end{proof}

\begin{remark} \emph{(1)} It is known that if we take
$k:=\mathbb{F}_{p}$ then the multiplication of the Witt vector group
scheme indeed corresponds to the additive of the $p$-adic numbers.
Theorem 5.1 gives us some hint that sometimes it is possible to
explain the addition of the $p$-adic numbers through the
comultiplication of path coalgebras.

\emph{(2)} For any $L(n,d)$, there is still one thing which is not
clear to us. That is, we don't know how to give the expression of
each path through generators although we can give in some special
cases (see the example below).

\emph{(3)} Not all Hopf structures over $k\circlearrowleft_{p^{n}}$
for some $n\geq 2$ are always commutative. In fact, set $p=2$ and
consider the associative algebra $H$ generated by $x,y$ with
relations
$$xy-yx=x,\;\;x^{2}=y^{2}=0.$$
Define the comultiplication $\Delta$, counit $\varepsilon$ and the
antipode through $$\Delta(x)=1\otimes x+x\otimes
1,\;\;\Delta(y)=1\otimes y+y\otimes 1+x\otimes x,$$
$$\varepsilon(x)=\varepsilon(y)=0,\;\;S(x)=-x,\;\;S(y)=-y.$$
It is straightforward to show that $H$ is indeed a Hopf algebra over
the path coalgebra $k\circlearrowleft_{p^{2}}$. Clearly, it is not
commutative.
\end{remark}

\begin{example} \emph{For the Hopf algebra $L(2,1)$, one can see that up to a coalgebra automorphism}
$$\alpha_{sp+t}=\frac{1}{s!t!}\alpha_{p}^{s}\alpha_{1}^{t},\;\;\emph{\textrm{for}}\;0\leq s\leq p-1,\;0\leq t\leq p-1.$$
\emph{In fact, we can prove this by using induction on the lengths
of pathes. If the length is 1, it is clear. Now assume it is true
for the pathes with lengths not more than $sp+t$. Now we consider
the case $sp+t+1$. To show this case, begin with an observation at
first. For any element $p$ in $k\circlearrowleft$, one always have}
$$\Delta(p)=\sum_{i=0}^{n} \alpha_{i}\otimes p_{(i)}$$
\emph{where $p_{(i)}$ are uniquely determined since
$1,\alpha_{1},\alpha_{2},\ldots$ is a basis of $k\circlearrowleft$.
For two elements $p,q$, the basic observation is, up to a coalgebra
automorphism,}
$$p=q\;\;\;\;\emph{\textrm{if and only if}} \;\;\;\;p_{(1)}=q_{(1)}.\;\;\;(\star)$$

\emph{Now we consider the case $sp+t+1$. If $0<t< p-1$, we just need
to show that
$\alpha_{sp+t+1}=\frac{1}{s!(t+1)!}\alpha_{p}^{s}\alpha_{1}^{t+1}$.
By $(\star)$, it is enough to show that
$\alpha_{sp+t}=(\frac{1}{s!(t+1)!}\alpha_{p}^{s}\alpha_{1}^{t+1})_{(1)}$.
Note that by assumption
$\frac{1}{s!(t+1)!}\alpha_{p}^{s}\alpha_{1}^{t+1}=\frac{1}{t+1}\alpha_{sp+t}\alpha_{1}$
and direct computation shows that
$(\alpha_{sp+t}\alpha_{1})_{(1)}=(t+1)\alpha_{sp+t}$.}

\emph{If $t=0$, we need show that
$\alpha_{sp+1}=\frac{1}{s!}\alpha_{p}^{s}\alpha_{1}=\alpha_{sp}\alpha_{1}$
by assumption. Clearly, $(\alpha_{sp+1})_{(1)}=\alpha_{sp}$ and
$(\alpha_{sp}\alpha_{1})_{(1)}=\alpha_{(s-1)p+(p-1)}\alpha_{1}+\alpha_{sp}$.
Note that in $L(2,1)$, $\alpha_{1}^{p}=0$ and so
$\alpha_{(s-1)p+(p-1)}\alpha_{1}=0$. By $(\star)$ again,
$\alpha_{sp+1}=\frac{1}{s!}\alpha_{p}^{s}\alpha_{1}$. }

\emph{If $t=p-1$, the equality that we need check is
$\alpha_{(s+1)p}=\frac{1}{(s+1)!}\alpha_{p}^{s+1}$. Also,
computations show that
$(\alpha_{p}^{s+1})_{(1)}=(s+1)\alpha_{p}^{s}\alpha_{p-1}=(s+1)\frac{1}{(p-1)!}\alpha_{p}^{s}\alpha_{1}^{p-1}$
and so
$(\frac{1}{(s+1)!}\alpha_{p}^{s+1})_{(1)}=\frac{1}{s!(p-1)!}\alpha_{p}^{s}\alpha_{1}^{p-1}$.
Meanwhile,
$(\alpha_{(s+1)p})_{(1)}=\alpha_{sp+(p-1)}=\frac{1}{s!(p-1)!}\alpha_{p}^{s}\alpha_{1}^{p-1}$
by assumption. Using  $(\star)$ again,
$\alpha_{(s+1)p}=\frac{1}{(s+1)!}\alpha_{p}^{s+1}$.}

\end{example}

\vskip 0.5cm

\noindent{\bf Acknowledgements:} The research was supported by the
NSF of China (10601052, 10801069, 10971206). The second author is
supported by Japan Society for the Promotion of Science under the
item ``JSPS Postdoctoral Fellowship for Foreign Researchers" and
Grand-in-Aid for Foreign JSPS Fellow. He thanks Professor A. Masuoka
for stimulating discussions. Some ideas was gotten during the second
and the third authors visited Chen Institute of Mathematics and they
thank Professor Cheng-Ming Bai for his warm-hearted helping. Part of
this work was done when the first and second authors visited the
University of Cologne under the financial support from DAAD. They
would also like to thank their host Professor Steffen K\"{o}nig for
his kind hospitality.

\end{document}